\documentclass[12pt]{amsart}
\usepackage{amsmath}
\usepackage{amssymb}
\usepackage{latexsym}
\usepackage[all]{xy}
\usepackage{caption}
\usepackage{graphicx,psfrag}
\usepackage{multirow}
\usepackage{color}
\usepackage{epstopdf}

\newtheorem{tpcl}{Tpcl}[section]
\newtheorem{Lemma}[tpcl]{Lemma}
\newtheorem{Theorem}[tpcl]{Theorem}
\newtheorem{Proposition}[tpcl]{Proposition}

\newtheorem{Definition}[tpcl]{Definition}
\newtheorem{Example}[tpcl]{Example}
\newtheorem{Remark}[tpcl]{Remark}

\begin{document}

\title[$\mathbb{Z}_2$-Thurston norms and pseudo-horizontal surfaces]{On $\mathbb{Z}_2$-Thurston norms and pseudo-horizontal surfaces in orientable Seifert $3$-manifolds}

\author{Xiaoming Du}
\address{School of Mathematics, South China University of Technology,
Guangzhou, 510640, PR China}
\email{scxmdu@scut.edu.cn}
\thanks{This research is supported by Fundamental Research Funds for the Central Universities in China. The author would like to thank Stephan Tillmann and Hyam Rubinstein for valuable comments. The author also thanks Adam Levine, Zhongtao Wu, Jingling Yang for their helps on the calculation of the $d$-invariant or correction term. The author is grateful to Xuezhi Zhao and the referee for their help to clarify many statements and proofs. }

\date{}


\maketitle

\begin{abstract}
We describe a general method to compute the $\mathbb{Z}_2$-Thurston norm for every $\mathbb{Z}_2$-homology class in an orientable Seifert manifold with orientable orbit surface. Our main tools are pseudo-horizontal surfaces. We give a necessary and sufficient criterion for the existence of pseudo-horizontal surfaces, calculate the non-orientable genera for such surfaces, and detect their $\mathbb{Z}_2$-homology classes. We then describe an algorithm to calculate the $\mathbb{Z}_2$-Thurston norm of each $\mathbb{Z}_2$-homology classes. We also present several interesting examples.
\end{abstract}

\section{Introduction}

Let $M$ be a closed orientable irreducible connected 3-manifold and $\alpha$ be a non-zero element in $H_2(M;\mathbb{Z}_2)$. Let $S$ be a closed embedded surface representing $\alpha$. Such an $S$ can be non-orientable. It can also be disconnected. We can define the \textbf{$\mathbb{Z}_2$-Thurston norm} of $\alpha$, i.e.,
$$
\| \alpha \|_{\mathbb{Z}_2-Th} = \underset{[S] = \alpha }{\min} \left\{ \sum_i \max \{0, -\chi(S_i)\} \right\},
$$
where $S$ varies over all surfaces that represent $\alpha$, $S_i$'s are the components of $S$, $\chi(S_i)$ is the Euler characteristic, and the summation is over all components. A lower bound for the complexity of $M$ (i.e., the minimum number of tetrahedra in the pseudo-simplicial triangulations of the $M$) is given by calculating the $\mathbb{Z}_2$-Thurston norm of every non-zero element in a rank one or two subgroup of $H_2(M;\mathbb{Z}_2)$, see \cite{JRT,JRST}. The surface $S$ representing $\alpha$ is said to be \textbf{$\mathbb{Z}_2$-taut} if (1) no component of $S$ is a sphere or projective plane and (2) $\chi(S) = -\| \alpha \|_{\mathbb{Z}_2-Th}$.

When $S$ is $\mathbb{Z}_2$-taut, as mentioned in \cite{JRT}, every connected component of $S$ must be geometrically incompressible, i.e. it does not have any compressing disk. Conversely, for a closed embedded non-orientable surface $S$ representing $\alpha$ ($S$ can be disconnected), how do we know whether it is geometrically incompressible? The only `easy' way is to check if it is $\mathbb{Z}_2$-taut. We want to determine the $\mathbb{Z}_2$-Thurston norm for every element in $H_2(M;\mathbb{Z}_2)$. When $S$ is a closed embedded connected non-orientable $\mathbb{Z}_2$-taut surface, the $\mathbb{Z}_2$-Thurston norm of $[S]=\alpha$ is $g(S)-2$, where $g(S)$ is the non-orientable genus of $S$, i.e., the number of $\mathbb{R}P^2$'s in the connected sum decomposition of $S$. The non-orientable genus of such an $S$ is minimal among all connected embedded non-orientable surfaces representing $\alpha$. So in this paper when we detect whether a closed connected embedded non-orientable surface is $\mathbb{Z}_2$-taut, we use both the non-orientable genus and the Euler characteristic of the surface.

The problem of determining the minimal non-orientable genus of an embedded non-orientable surface representing a given $\mathbb{Z}_2$-homology class has a long history. Bredon and Wood \cite{BW} solved this problem for lens spaces. Ni and Wu \cite{NW} gave a lower bound to the genus by using Heegaard Floer homology. Levine, Ruberman and Strle pointed out that the lower bound is not sharp and gave an example (\cite[Proposition 7.4]{LRS}). We use a more geometric and elementary method to calculate the exact value of the $\mathbb{Z}_2$-Thurston norms for orientable Seifert manifolds.

Our calculation depends on special types of closed embedded non-orientable surfaces in orientable Seifert manifolds. They are pseudo-vertical surfaces and pseudo-horizontal surfaces defined by Frohman \cite{Fr}. In an orientable Seifert manifold, for the non-orientable surfaces, Frohman has proved that a geometrically incompressible closed connected embedded non-orientable surface is isotopic to a pseudo-vertical surface or a pseudo-horizontal surface. For the orientable surfaces, it is known that an incompressible closed connected embedded orientable surface is isotopic to a vertical surface or a horizontal surface (see \cite[Proposition 1.11]{Ha}). Since the $\mathbb{Z}_2$-taut surface is geometrically incompressible, for a given non-zero element in $H_2(M;\mathbb{Z}_2)$, we only need to find all vertical, horizontal, pseudo-vertical, and pseudo-horizontal surfaces that represent such an element and compare their Euler characteristics.

Vertical surfaces and horizontal surfaces have been classified (see \cite[Section 2.1]{Ha}). Pseudo-vertical surfaces have been studied by Frohman \cite{Fr}. There is no literature on pseudo-horizontal surfaces. We will study pseudo-horizontal surfaces systematically, calculate their genera, and detect their $\mathbb{Z}_2$-homology classes. Their genera rely on a recursive function $N(2k,q)$ which is defined in \cite{BW}, also see Section 2. Our calculation shows that in many cases the pseudo-vertical surfaces are $\mathbb{Z}_2$-taut, but sometimes they are not. We find several families of Seifert manifolds. In each of these manifolds, there exists a pseudo-horizontal surface whose genus is less than the genus of the pseudo-vertical surface in the same $\mathbb{Z}_2$-homology class. The difference of their genera is 2. We also find an example that a pseudo-vertical surface is isotopic to a pseudo-horizontal surface (recall that in Seifert manifolds vertical surfaces are not isotopic to horizontal surface).

We organize this paper as follows. In Section 2, we illustrate the constructions of pseudo-vertical surfaces and pseudo-horizontal surfaces. In Section 3, we give a criterion for the existence for pseudo-horizontal surfaces in orientable Seifert fibered manifolds. We also calculate the genera of the pseudo-horizontal surfaces. In Section 4, we discuss the $\mathbb{Z}_2$-homology groups of the orientable Seifert manifolds and represent the $\mathbb{Z}_2$-homology classes by pseudo-vertical surfaces and pseudo-horizontal surfaces. We also show how to detect the $\mathbb{Z}_2$-homology class for a given pseudo-horizontal surface. Actually we can use these to calculate the $\mathbb{Z}_2$-cohomology rings. In Section 5, we give several interesting examples of pseudo-horizontal surfaces. In Section 6, we study some properties of the function $N(2k,q)$. These properties play crucial roles in the calculation for the $\mathbb{Z}_2$-Thurston norms. In Section 7, we give a deterministic algorithm to calculate the $\mathbb{Z}_2$-Thurston norms for orientable Seifert manifolds.


\section{Pseudo-vertical and pseudo-horizontal surfaces}

\subsection{Notations for Seifert fibered manifolds.}

We only consider closed orientable Seifert fibered manifolds whose orbit space is orientable. We fix the following notations and recall some of the known results. For details, see \cite[Chapter 10]{Ma} or \cite[Chapter 2]{Ha}.
\begin{itemize}
  \item $S^1$: a circle.
  \item $\Sigma_g^n$: an orientable compact connected genus $g$ surface with $n$-punctures.
  \item $\Sigma_g$: a closed orientable connected genus $g$ surface.
  \item $M_0$: $\Sigma_g^n \times S^1$.
  \item $T_i (i = 1,\cdots,n)$: the $i$-th boundary of $\partial M_0$, homeomorphic to a torus.
  \item $v_i (i = 1,\cdots,n)$: the simple closed curve on $T_i$ that goes along the $S^1$ direction. The letter ``v'' stands for ``vertical''.
  \item $h_i (i = 1,\cdots,n)$: the simple closed curve on $T_i$ that lies on a fixed horizontal section $\Sigma_g^n \times \{x\}$ and satisfying $[h_1] + \cdots + [h_n] = 0 \in H_1(M_0; \mathbb{Z})$. The letter ``h'' stands for ``horizontal''.
  \item $R_i (i = 1,\cdots,n)$: each $R_i$ is a solid torus.
  \item $m_i (i = 1,\cdots,n)$: the simple closed curve on $\partial R_i$ that bounds a disk in $R_i$. The letter ``m'' stands for ``meridian''.
  \item $l_i (i = 1,\cdots,n)$: a simple closed curve on $\partial R_i$ that intersects with $m_i$ once. The letter ``l'' stands for ``longitude''.
  \item $f_i (i = 1,\cdots,n)$: the identifying map $\partial R_i \to \partial M_0$.
  \item $\alpha_i,\beta_i,\gamma_i,\delta_i (i = 1,\cdots,n)$: integers satisfying
    $$
    (f_i)_*:  \left( \begin{matrix} [m_i] \\ [l_i] \end{matrix}\right)  \mapsto
    \left(\begin{matrix} \alpha_i & \beta_i \\ \gamma_i & \delta_i \end{matrix}\right)\left( \begin{matrix} [h_i] \\ [v_i] \end{matrix}\right),
    $$
    $$
    (f_i^{-1})_*:  \left( \begin{matrix} [h_i] \\ [v_i] \end{matrix}\right)  \mapsto
    \left(\begin{matrix} \delta_i & -\beta_i \\ -\gamma_i & \alpha_i \end{matrix}\right)\left( \begin{matrix} [m_i] \\ [l_i] \end{matrix}\right),
    $$
    where $\alpha_i \delta_i - \beta_i \gamma_i = 1$ and $\alpha_i \neq 0$.
  \item $M$: $M = M_0 \bigcup_{f_1} R_1 \bigcup_{f_2} \cdots \bigcup_{f_n} R_n$.
\end{itemize}

$M_0$ can be fibered along the $S^1$ direction. The fibration can extend into the solid tori $R_i$'s. Inside $R_i$ there is a singular fiber with multiplicity $\alpha_i$. The gluing result $M$ has a Seifert fibered structure. Following the notation in \cite{Ma}, we denote it by $(\Sigma_g,(\alpha_1,\beta_1),\cdots,(\alpha_n,\beta_n))$.

In the literatures, such a $(\Sigma_g,(\alpha_1,\beta_1),\cdots,(\alpha_n,\beta_n))$ has different notations. In Hatcher \cite{Ha}, its notation is $M(+g,0;\beta_1/\alpha_1,\cdots,\beta_n/\alpha_n)$. In Orlik \cite{Or}, its notation is $[e,(o_1,0);(\alpha_1,\beta'_1),\cdots,(\alpha_n,\beta'_n)]$, where (1) $e$ is an integer satisfying $e + \sum \beta'_i/\alpha_i = \sum \beta_i/\alpha_i$ and (2) for each $i$ we have $\alpha_i>0$ and $0 < \beta_i' < \alpha_i$. In some literatures such as \cite{BHZZ} it is also denoted by $(O,o,g \mid e:(\alpha_1,\beta_1'),\cdots,(\alpha_n,\beta_n'))$, where $e,\beta'_i$ are the same as above.

\begin{Remark}\label{rem-Seifert-notation}
In the notation of this paper, we require $\alpha_i \geqslant 2$ for each $i$. If there is some $\alpha_i=1$ (for simplicity, $\alpha_1=1$), then we cancel $(\alpha_1,\beta_1)$ by the following rule:
$$
\begin{array}{rcl}
& & (\Sigma_g,(1,\beta_1),(\alpha_2,\beta_2),(\alpha_3,\beta_3),\cdots) \\
&\cong& (\Sigma_g,(1,0),(\alpha_2,\beta_2 + \beta_1 \cdot \alpha_2),(\alpha_3,\beta_3),\cdots) \\
&\cong& (\Sigma_g,(\alpha_2,\beta_2 + \beta_1 \cdot \alpha_2),(\alpha_3,\beta_3),\cdots)
\end{array}
$$
\end{Remark}

\subsection{Geometrically incompressible surfaces in a solid torus}

Suppose $R$ is a solid torus $D^2 \times S^1$, $m$ is the meridian circle on $\partial R$, $l$ is a simple closed curve on $\partial R$ that intersect $m$ once.
Bredon and Wood \cite[Section 6,7,8]{BW} and Rubinstein \cite[Theorem 13]{Ru} treated the geometrically incompressible surfaces (which is non-orientable and with one boundary) in $R$. Their results are as follows.

\begin{Proposition}[Bredon-Wood, Rubinstein]\label{prop-Bredon-Wood}
Let $R,m,l$ be as previous.
\begin{enumerate}
  \item If $S$ is a geometrically incompressible surface in $R$, $\partial S$ lies on $\partial R$, and $[\partial S] = p [l] + q [m]$ ($p$ is coprime to $q$), then $p$ must be even.
  \item On $\partial R$, suppose $c$ is a simple closed curve whose isotopy class is $2k [l] + q [m]$ ($k \neq 0$, $q$ is coprime to $2k$). Then inside $R$, up to isotopy, there exists a unique geometrically incompressible non-orientable surface $S$ with the boundary curve as $c$.
\end{enumerate}
\end{Proposition}

\begin{Proposition}[Bredon-Wood]\label{prop-value-of-N}
Under the same conditions and notations as the previous proposition, we denote the genus of the geometrically incompressible non-orientable surface by $N(2k,q)$. If $2k>q>0$, then $N(2k,q)$ is given by the following two equivalent ways:
\begin{enumerate}
  \item Recursive formula: $N(2k,1)=k$, $N(2k,q) = N(2(k-Q),q-2m)+1$, where $Q,m$ satisfy $2km - Qq = \pm 1$ and $0 < Q < k$.
  \item Suppose $2k/q$ can be written as a continued fraction
  $$
  \begin{array}{rcl}
    2k/q
    & = & [a_0,a_1,\dots,a_n] \\
    & = & a_0 + \cfrac{1}{a_1 +
            \cfrac{1}{a_2 +
              \cfrac{1}{
                \ddots + \cfrac{1}{a_n}
                }
            }
          }
  \end{array}
  $$
  where $a_i$'s are integers, $a_0 \geq 0$, $a_i>0$ for $1 \leq i \leq n$, and $a_n>1$. Add $a_i$ successively except that when a partial sum is even we skip the next $a_i$. That is, define inductively $b_0 = a_0$,
  $$
  b_i = \left\{ \begin{array}{ll}a_i, & b_{i-1} = 0 \text{ or } \sum_{j=0}^{i-1} b_j \text{ is odd}, \\ 0, & b_{i-1} \neq 0 \text{ and } \sum_{j=0}^{i-1} b_j \text{ is even}. \end{array}\right.
  $$
  Then $N(2k,q) = \frac{1}{2} \sum_{i=0}^{n} b_i.$
\end{enumerate}
\end{Proposition}

\begin{Remark}
By the genus of a connected non-orientable surface with boundaries, we mean the non-orientable genus of the closed connected non-orientable surface obtained by capping off each component of the boundary with a disk.
\end{Remark}

\begin{Remark}
Each of these geometrically incompressible surfaces in the solid torus is boundary compressible. Repeated boundary compressions reduce it to a meridian disk.
\end{Remark}

\begin{Remark}
In the definition of $N(2k,q)$, the condition $k \neq 0$ can be removed if we define $N(0,1)=0$ since the meridian $m$ bounds a disk in $R$.
\end{Remark}

\begin{Remark}\label{rem-recursive-of-N}
In fact, $N(2k,q)$ is also the minimal genus of the closed connected embedded non-orientable surface in the lens space $L(2k,q)$ representing the non-zero $\mathbb{Z}_2$-homology class.
As mentioned in \cite{BW},
by the homeomorphism classification for lens space and other geometric constructions, the function $N(2k,q)$ also satisfies the following identities:
\begin{enumerate}
  \item $N(-2k,-q) = N(2k,q)$ (hence for general integers $k \neq 0$ and $q$, we can always normalize them such that $k>0$).
  \item $N(2k,q+2k) = N(2k,q)$ (hence for general integers $k$ and $q$ satisfying $k>0$, we can always normalize them such that $0<q<2k$).
  \item $N(2k,q) = N(2k,2k-q)$ (hence for general integers $k$ and $q$ satisfying $0<q<2k$, we can always normalize them such that $0<q<k$).
  \item $N(2k,q) = N(2k,q')$, where $0<q,q'<2k$, $qq' \equiv 1 \pmod{2k}$.
  \item $N(2k+2hq,q) = h + N(2k,q)$, where $0<q<2k, h>0$.
\end{enumerate}
The above properties together are in fact equivalent to the recursive formula of $N(2k,q)$ in Proposition \ref{prop-value-of-N}. For the detail of the equivalence, see \cite{BW}. Section 6 of our paper will discuss more properties of $N(2k,q)$.
\end{Remark}

\begin{Lemma}\label{lem-two-surfaces-inside-solid-torus}
Suppose: (1) $c_1, c_2$ are two parallel simple closed curves on $\partial R$ whose isotopy class is $2k[l] + q[m]$ ($k \neq 0$), and (2) $S_1$ and $S_2$ are two connected embedded non-orientable surfaces inside $R$ with $\partial S_1 = c_1, \partial S_2 = c_2$. Then $S_1 \bigcap S_2 \neq \varnothing$.
\end{Lemma}

\begin{proof}
The fundamental group of $R$ is $\mathbb{Z}$. Take the covering space $\widetilde{R}$ corresponding to the subgroup $2\mathbb{Z}$.
Denote the lifting of $S_1,S_2,m,l$ in $\widetilde{R}$ by $\widetilde{S}_1, \widetilde{S}_2, \widetilde{m}, \widetilde{l}$ respectively.
Then $\widetilde{S_1}$ is a two-sided surface. $\partial \widetilde{S}_1$ consists of two simple closed curves. $\widetilde{R} \setminus \widetilde{S}_1$ is disconnected. Both sides have points in $\partial \widetilde{S}_2$. Hence $\widetilde{S_1} \bigcap \widetilde{S_2} \neq \varnothing$.
\end{proof}

\subsection{Pseudo-vertical surfaces and pseudo-horizontal surfaces}

Suppose $M$ is an orientable Seifert fibered manifold with orientable orbit space. We follow the notation at the beginning of Section 2. The following definitions are taken from the literature, such as \cite{Ha} and \cite{Fr}.

\begin{Definition}\label{def-pseudo}
An embedded orientable surface in $M$ is called \textbf{vertical} if it consists of regular fibers. An embedded orientable surface in $M$ is called \textbf{horizontal} if it intersects each regular fiber. An embedded non-orientable surface $S$ in $M$ is called \textbf{pseudo-vertical} if: (1) $S \cap M_0$ is a vertical annulus whose boundary lies in two distinct $T_i$ and $T_j$, and (2) $S \cap R_i$ and $S \cap R_j$ are geometrically incompressible non-orientable surfaces inside the solid torus $R_i$ and $R_j$ respectively. For simplicity, we also called it the pseudo-vertical surface connecting $R_i$ and $R_j$. We denote such a pseudo-vertical surface by $V_{i,j}$. An embedded non-orientable surface $S$ in $M$ is called \textbf{pseudo-horizontal} if: (1) $S \cap M_0$ is horizontal in $M_0$, and (2) $S \cap R_i$ is either a family of meridian disks or a geometrically incompressible non-orientable surface in $R_i$ for each $i$.
\end{Definition}

For example, when there are some $i,j$ such that $(\alpha_i,\beta_i)=(\alpha_j,\beta_j)=(2,1)$, Figure 1 shows how a pseudo-vertical surface $V_{i,j}$ looks like in the Seifert fibered manifold $(\Sigma_g,\cdots,(2,1),\cdots,(2,1),\cdots)$. The upper side should be identified with the bottom side. Figure 1-(left) shows the local of $M_0$ and the vertical annulus $S \cap M_0$. Figure 1-(right) shows two M\"obius bands in solid tori $R_i,R_j$. The pseudo-vertical surface $V_{i,j}$ is a Klein bottle.
\begin{figure}[htbp]
\begin{center}
\includegraphics[height=6cm]{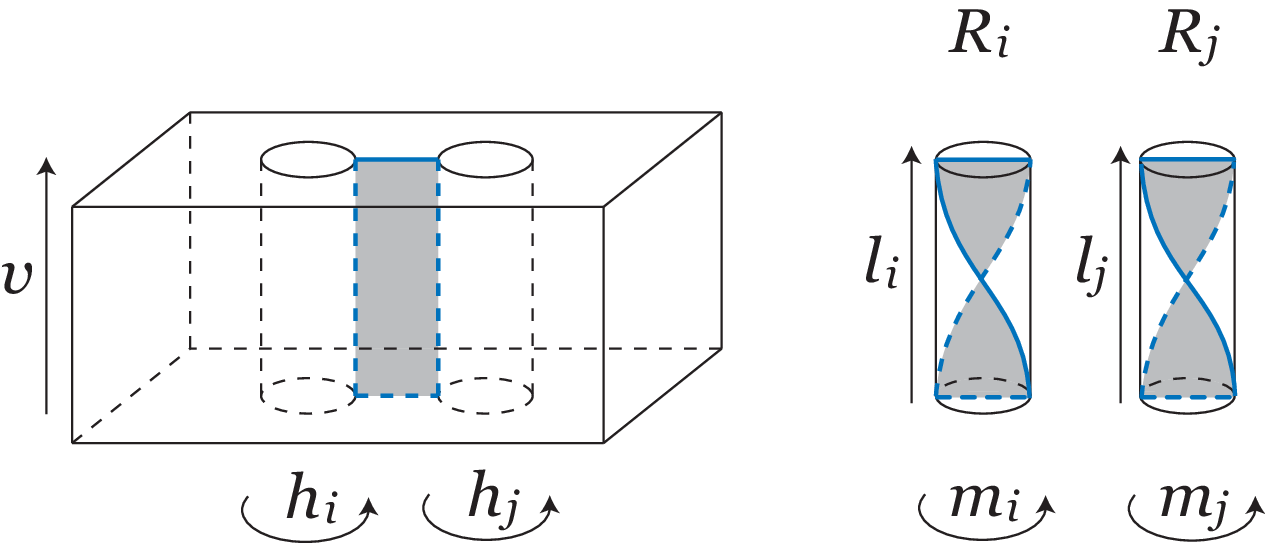}\\
Figure 1.
\end{center}
\end{figure}

\begin{Remark}
The pseudo-horizontal surfaces have a more geometric description. Jaco \cite{Ja} constructed what is now called ``staircase'' surfaces in $M_0 = \Sigma_g^n \times S^1$. Then a pseudo-horizontal surface is obtained by attaching either disks or the geometrically incompressible non-orientable surfaces in $R_i$'s to such a staircase surface. A more concrete picture in a specific Seifert fibered manifold is shown in Figure 3 of Example \ref{first-example} in our paper.
\end{Remark}

\begin{Lemma}\label{lem-genus-pseudo-vertical}
The existence of the pseudo-vertical surface $V_{i,j}$ requires that $\alpha_i,\alpha_j$ are even, and the genus of $V_{i,j}$ is $N(\alpha_i,-\gamma_i)+N(\alpha_j,-\gamma_j) = N(\alpha_i,\beta_i)+N(\alpha_j,\beta_j)$.
\end{Lemma}
\begin{proof}
Under the attaching map, the homology class of the regular fiber $[v]$ is mapped to the class $-\gamma_i[m_i]+\alpha_i[l_i]$ on $\partial R_i$. By Proposition \ref{prop-Bredon-Wood} and Proposition \ref{prop-value-of-N}, it can bound a geometrically incompressible surface inside $R_i$ if and only if $\alpha_i$ is even, and the geometrically incompressible surface has genus $N(\alpha_i,-\gamma_i)$. Since $\alpha_i\delta_i-\beta_i\gamma_i=1$, by Remark \ref{rem-recursive-of-N}, we have $N(\alpha_i,-\gamma_i)=N(\alpha_i,\beta_i)$. The situation for $R_j$ is similar.
\end{proof}

%
%
%

Using pseudo-vertical surfaces and pseudo-horizontal surfaces, Frohman proved structure theorems for the geometrically incompressible non-orientable surfaces in Seifert manifolds.

\begin{Theorem}\label{thm-standard-form-of-geometric-incompressible}\cite[Theorem 2.5]{Fr}
Every closed geometrically incompressible non-orientable surface in a Seifert fibered manifold $M$ is isotopic to a pseudo-vertical surface or a pseudo-horizontal surface.
\end{Theorem}

\begin{Theorem}\cite[Theorem 3.1]{Fr}
If the Seifert fibered manifold $M$ has at least 4 singular fibers or the orientable base surface has positive genus, then every pseudo-vertical surface in $M$ is geometrically incompressible.
\end{Theorem}

\begin{Remark}
When the base surface is a sphere and the Seifert fibered manifold has 3 singular fibers, the pseudo-vertical surfaces are not always geometrically incompressible. In \cite[Section 3]{Fr} Frohman gave the example of $S^2((2,-1),(3,1),(6,1))$.
\end{Remark}

\section{The conditions for the existence of pseudo-horizontal surfaces in orientable Seifert manifolds}

We still use the notation at the beginning of Section 2 and fix the notation for pseudo-horizontal surfaces in orientable Seifert manifolds as follows.
\begin{itemize}
  \item $Z$: a pseudo-horizontal surface in $M$.
  \item $Z_0$: $Z_0 = Z \bigcap M_0$. There is a covering map from $Z_0$ to $\Sigma_g^n$.
  \item $c_i (i = 1,\cdots,n)$: one of the several parallel simple closed curves of $\partial Z_0 \bigcap T_i$.
  \item $\lambda_i,\mu_i (i = 1,\cdots,n)$: integers satisfying $[c_i] = \left(\begin{matrix} \lambda_i & \mu_i  \end{matrix}\right)\left( \begin{matrix} [h_i] \\ [v_i] \end{matrix}\right)$. Since $(\lambda_i,\mu_i)$ and $(-\lambda_i,-\mu_i)$ correspond to the same curve, we can always assume $\lambda_i>0$. But $\mu_i$ is allowed to be $<0$. We also call $Z$ ``\textbf{the pseudo-horizontal surface determined by $(\lambda_1,\mu_1),\cdots,(\lambda_n,\mu_n)$}''.
  \item $\lambda$: the least common multiple of $\lambda_1,\dots,\lambda_n$. It equals to the covering degree of $Z_0 \to \Sigma_g^n$. The number of the components of $\partial Z_0 \cap T_i$ is $\lambda/\lambda_i$.
\end{itemize}

\begin{Proposition}\label{prop-parameters-identity}
We follow the notations at the beginnings of Section 2 and Section 3.
We have $\mu_1/\lambda_1 + \cdots + \mu_n/\lambda_n = 0$.
\end{Proposition}

\begin{proof}
The intersection $Z_0=Z \cap M_0$ is a horizontal surface in $M_0$. The boundary slopes of the components of $\partial Z_0$ on $\partial M_0$ are $(\lambda_i,\mu_i)\;(i=1,\cdots,n)$. If we attach solid tori to $M_0$ such that the meridians are glued to $(\lambda_i,\mu_i)$-curves, then we get a horizontal surface without boundary in the Seifert fibered manifold $(\Sigma_g,(\lambda_1,\mu_1),\cdots,(\lambda_n,\mu_n))$. The identity is given by the properties of horizontal surfaces, cf \cite{Wa}, \cite[Proposition 2.2]{Ha}, or \cite[Proposition 10.4.8]{Ma}.
\end{proof}


\begin{Proposition}\label{prop-lcm-restriction}
We follow the notations at the beginnings of Section 2 and Section 3.
For each $i$, one of the following two cases happens.
\begin{enumerate}
  \item $(\lambda_i,\mu_i) = (\alpha_i,\beta_i)$.
  \item $\lambda_i$ equals $\lambda$ (recall $\lambda$ is the least common multiple of $\lambda_1,\cdots,\lambda_n$).
\end{enumerate}
Moreover, there exists some $i$ such that $(\lambda_i,\mu_i)\neq(\alpha_i,\beta_i)$.
\end{Proposition}

\begin{proof}
If $\lambda/\lambda_i>1$, then $Z \bigcap R_i$ must be $\lambda/\lambda_i$ copies of geometrically incompressible surfaces inside $R_i$ and bounded by $\lambda/\lambda_i$ simple closed curves that parallel to $c_i$. If $c_i$ is not the meridian, by Lemma \ref{lem-two-surfaces-inside-solid-torus}, these surfaces must intersect and $Z$ is not embedded. If $(\lambda_i,\mu_i)=(\alpha_i,\beta_i)$ for each $i$, then we will get an orientable horizontal surface, not a non-orientable pseudo-horizontal surface.
\end{proof}


\begin{Proposition}\label{prop-parity}
We follow the notations at the beginnings of Section 2 and Section 3. For each $i$, we have $2 \mid (\lambda_i-\alpha_i)$ and $2 \mid (\mu_i - \beta_i)$.
\end{Proposition}

\begin{proof}
If $\lambda/\lambda_i>1$, then $(\lambda_i,\mu_i) = (\alpha_i,\beta_i)$, the statement of the proposition holds.
If $\lambda/\lambda_i=1$, then the homology class of $c_i$ on $\partial R_i$ is $(\lambda_i\delta_i-\mu_i\gamma_i)[m_i] + (-\lambda_i\beta_i+\mu_i\alpha_i)[l_i]$. It can bound a geometrically incompressible surface if and only if $-\lambda_i\beta_i+\mu_i\alpha_i$ is even. Since there is at most one of $\alpha_i,\beta_i$ is even, and at most one of $\lambda_i,\mu_i$ is even, we have one of the following cases happens:
\begin{enumerate}
  \item $\lambda_i$ odd, $\beta_i$ even, $\mu_i$ even, $\alpha_i$ odd.
  \item $\lambda_i$ even, $\beta_i$ odd, $\mu_i$ odd, $\alpha_i$ even.
  \item $\lambda_i$ odd, $\beta_i$ odd, $\mu_i$ odd, $\alpha_i$ odd.
\end{enumerate}
In each case, the statement of the proposition holds.
\end{proof}

\begin{Theorem}\label{thm-existence-for-pseudo-horizontal-surface}
There exists a pseudo-horizontal surface determined by $(\lambda_1,\mu_1)$, $\cdots$, $(\lambda_n,\mu_n)$ in a Seifert manifold $M = (\Sigma_g,(\alpha_1,\beta_1),\cdots,(\alpha_n,\beta_n))$ if and only if the results from Proposition \ref{prop-parameters-identity} to Proposition \ref{prop-parity} hold.
\end{Theorem}

\begin{proof}
The ``only if'' part is the above propositions. For the ``if'' part, since $(\alpha_i,\beta_i)$, $(\lambda_i,\mu_i)\;(i=1,\cdots,n)$ satisfy these conditions, we can first construct a horizontal surface in $M_0$ and then extend it into each $R_i$ to construct the pseudo-horizontal surface.
\end{proof}

\begin{Theorem}\label{thm-genus-of-pseudo-horizontal-surface}
We follow the notations at the beginnings of Section 2 and Section 3.
The non-orientable genus of $Z$ is
$$
2 + \lambda\left(n-2+2g - \sum_{i=1}^n \frac{1}{\lambda_i}\right) + \sum_{i=1}^n N(-\lambda_i \beta_i+\mu_i \alpha_i,\lambda_i\delta_i-\mu_i\gamma_i),
$$
where $N(\cdot,\cdot)$ is as in Proposition \ref{prop-value-of-N}.
\end{Theorem}

\begin{proof}
Suppose $Z_0 = Z \cap M_0$. If we cap off all the boundaries of $Z_0$ with disks, then the Euler characteristic of the resulting closed orientable surface can be computed by the Riemann-Hurwitz formula as
$$
\lambda \left[(2-2g) - \sum_{i=1}^n (1-\frac{1}{\lambda_i})\right] = \lambda\left(\sum_{i=1}^n\frac{1}{\lambda_i}-n-2g+2\right).
$$
Its genus is two minus the Euler characteristic. Then we replace the disks by the embedded non-orientable surfaces in $R_i$'s to obtain the result.
\end{proof}

\section{Representing the elements in the $\mathbb{Z}_2$-homology groups by geometric closed curves and surfaces}

We follow the notations at the beginnings of Section 2 and Section 3.
$v_1\cdots,v_n$ are in the same homology class. We denote their homology class by $[v]$.

\begin{Lemma}\label{lem-fundamental-group-presentation}
Suppose $M = (\Sigma_g,(\alpha_1,\beta_1),\cdots,(\alpha_n,\beta_n))$. Then we have
$$
\begin{array}{rl} \pi_1(M) = \langle & a_1,b_1,\cdots,a_g,b_g,h_1,\cdots,h_n,v  \mid  \\
& \quad a_iv=va_i,b_iv=vb_i\;(i=1,\cdots,g),\\
& \quad h_jv=vh_j,h_j^{\alpha_j}v^{\beta_j}=1\;(i=1,\cdots,n),\\
& \quad h_1 \cdots h_n[a_1,b_1]\cdots[a_g,b_g]=1 \quad \rangle. \end{array}
$$
\end{Lemma}
\begin{proof}
See \cite[Section 5.3]{Or}.
\end{proof}

\begin{Lemma}\label{lem-homology-group-presentation}
Suppose $M = (\Sigma_g,(\alpha_1,\beta_1),\cdots,(\alpha_n,\beta_n))$. Then we have
$$\begin{array}{rl} H_1(M;\mathbb{Z}_2) = \langle & [a_1],[b_1],\cdots,[a_g],[b_g],[h_1],\cdots,[h_n],[v]  \mid  \\
& \quad \alpha_i[h_i] + \beta_i[v] = 0\;(i=1,\cdots,n),\\
& \quad [h_1] + \cdots + [h_n] = 0 \quad \rangle. \end{array}$$
\end{Lemma}
\begin{proof}
Take the abelization of the fundamental group.
\end{proof}

\begin{Lemma}\label{lem-homology-group-generator}
Suppose $M = (\Sigma_g,(\alpha_1,\beta_1),\cdots,(\alpha_n,\beta_n))$.
\begin{enumerate}
  \item[(1)] When each of $\alpha_1,\cdots,\alpha_n$ is odd,
    \begin{itemize}
      \item[(1.1)] if $\beta_1 + \cdots + \beta_n$ is odd, then $H_1(M;\mathbb{Z}_2)=\mathbb{Z}_2^{2g}$, which is generated by $\{[a_i],[b_i] \mid i=1,\cdots,g \}$.
      \item[(1.2)] if $\beta_1 + \cdots + \beta_n$ is even, then we can use fiber moves to make all $\beta_i$'s even, and $H_1(M;\mathbb{Z}_2) = \mathbb{Z}_2^{2g} \oplus \mathbb{Z}_2$, which is generated by $\{[v], [a_i],[b_i] \mid i=1,\cdots,g \}$.
    \end{itemize}
  \item[(2)] When there are $k>0$ even numbers among $\alpha_1,\cdots,\alpha_n$, we have $H_1(M;\mathbb{Z}_2) = \mathbb{Z}_2^{2g} \oplus \mathbb{Z}_2^{k-1}$. The first direct sum part $\mathbb{Z}_2^{2g}$ is generated by $\{[a_i],[b_i] \mid i=1,\cdots,g\}$. The second direct sum part is generated by $\{[h_i] \mid \alpha_i \text{ mod } 2 = 0\}$ and $\sum_{\alpha_i \text{ mod } 2 = 0} [h_i] = 0$.
\end{enumerate}
\end{Lemma}

\begin{proof}
Direct computation.
\end{proof}

By duality, $H_2(M;\mathbb{Z}_2) \cong H_1(M;\mathbb{Z}_2)$.

\begin{Lemma}\label{lem-homology-class-vertical}
We follow the notations at the beginnings of Section 2, the notations of Definition \ref{def-pseudo}, and the notations of Lemma \ref{lem-homology-group-presentation}. If $\alpha_i,\alpha_j$ are even, then we can compute the $\mathbb{Z}_2$-intersection number of the pseudo-vertical surfaces $V_{i,j}$ and the basis of $H_1(M;\mathbb{Z}_2)$ as follows:
$[V_{i,j}]\cdot[a_k]=[V_{i,j}]\cdot[b_k]=0\;(k=1,\cdots,g)$, $[V_{i,j}]\cdot[h_i]=[V_{i,j}]\cdot[h_j]=1$, for other $[h_k]$ we have $[V_{i,j}]\cdot[h_k]=0$.
\end{Lemma}
\begin{proof}
Direct geometric construction. For a picture, see Figure 1.
\end{proof}

\begin{Lemma}\label{lem-represent-homology-class}
Suppose $M = (\Sigma_g,(\alpha_1,\beta_1),\cdots,(\alpha_n,\beta_n))$. Each of the non-zero elements in $H_2(M;\mathbb{Z}_2)$ can be represented by an embedded surface (not necessary connected). Every connected component of the representing surface is a vertical, horizontal, pseudo-vertical, or pseudo-horizontal surface.
\end{Lemma}

\begin{proof}
We prove the lemma by two steps.

Step 1. We decompose $H_2(M;\mathbb{Z}_2)$ into a direct sum of two parts. The first direct sum part is $\mathbb{Z}_2^{2g}$. We will prove every non-trivial element in the first direct sum part can be represented by a vertical torus. The second part is $\{0\}$, $\mathbb{Z}_2$, or $\mathbb{Z}_2^{k-1}$ according to the cases in Lemma \ref{lem-homology-group-generator}. We will prove every non-trivial element in the second direct sum part can be represented by (1) a horizontal surface, (2) a pseudo-horizontal surface, or (3) an embedded surface (not necessary connected) whose connected components are pseudo-vertical surfaces.

Step 2. Now we have $H_2(M;\mathbb{Z}_2)\cong \mathbb{Z}_2^{2g} \oplus H$. Given an element in $H_2(M;\mathbb{Z}_2)$, project it to $\mathbb{Z}_2^{2g}$ and $H$ and construct the representing surfaces of the two projections. If these surfaces do not intersect, then take their union. If these surfaces intersect, then take their union and do surgeries on intersecting circles. The result of the surgeries is a horizontal surface or pseudo-horizontal surface.

Proof of Step 1.

By \cite[Section 6.2]{FM}, every non-zero element in $H_1(\Sigma_g;\mathbb{Z}_2) \cong \mathbb{Z}_2^{2g}$ can be represented by a simple closed curve $c$ on $\Sigma_g$. The product of the curve $c$ and $S^1$ is a vertical torus and represents a non-zero element in the first direct sum part $\mathbb{Z}_2^{2g}$ of $H_2(M;\mathbb{Z}_2)$. We have the $\mathbb{Z}_2$-intersection numbers $[c \times S^1] \cdot [v] = 0$, $[c \times S^1] \cdot [h_i] = 0$, $i = 1,\cdots,n$. About the second direct sum part of $H_2(M;\mathbb{Z}_2)$, we discuss following the cases of Lemma \ref{lem-homology-group-generator}.

Case (1.1). Each of $\alpha_1,\cdots,\alpha_n$ is odd and $\sum \beta_i$ is odd. $H_2(M;\mathbb{Z}_2) = \mathbb{Z}_2^{2g}$. There is no second direct sum part.

Case (1.2). Each of $\alpha_1,\cdots,\alpha_n$ is odd and $\sum \beta_i$ is even. $H_2(M;\mathbb{Z}_2) = \mathbb{Z}_2^{2g} \oplus \mathbb{Z}_2$. After fiber moves, we can assume each $\beta_i$ is even. For $1 \leq i \leq n-1$ we take $(\lambda_i,\mu_i)=(\alpha_i,\beta_i)$. Then we determine $(\lambda_n,\mu_n)$ from $-\mu_n/\lambda_n = \mu_1/\lambda_1 + \cdots + \mu_{n-1}/\lambda_{n-1}$. If $(\lambda_n,\mu_n)=(\alpha_n,\beta_n)$, then we can construct a horizontal surface. If $(\lambda_n,\mu_n)\neq(\alpha_n,\beta_n)$, we still have $2 \mid (\lambda_n-\alpha_n)$ and $2 \mid (\mu_n - \beta_n)$. By Theorem \ref{thm-existence-for-pseudo-horizontal-surface} there exists a pseudo-horizontal surface $Z$ determined by $(\lambda_1,\mu_1),\cdots,(\lambda_n,\mu_n)$. We can easily get the $\mathbb{Z}_2$-intersection number of $[Z]$ and the elements in $H_1(M;\mathbb{Z}_2)$ as follows: $[Z] \cdot [v]=1, [Z]\cdot [a_i]=0, [Z]\cdot [b_i]=0\;(i=1,\cdots,g)$. Hence $[Z]$ does not lie in the first direct sum part $\mathbb{Z}_2^{2g}$.

Case (2). Suppose there are $k>0$ even numbers among $\alpha_1,\cdots,\alpha_n$. For simplicity, assume $\alpha_1,\cdots,\alpha_k$ are even and $\alpha_{k+1},\cdots, \alpha_{n}$ are odd. For distinct (unordered) $2m$ numbers $\{i_1,j_1,\cdots,i_m,j_m\} \subset \{1,\cdots,k\}$, we group them in pairs, construct pseudo-vertical surfaces $V_{i_1,j_1},\cdots,V_{i_m,j_m}$ (the picture is like Figure 1), and then take the union surface, denoted by $V_{i_1,j_1,\cdots,i_m,j_m}$. The surfaces obtained by different ways of pairing will represent the same $\mathbb{Z}_2$-homology class.
Let $m$ vary from $1$ to $\lfloor \frac k2 \rfloor$. We can build
$C_k^2 + C_k^4 + \cdots + C_k^{2 \lfloor \frac k2 \rfloor} =  2^{k-1}-1$ such $V_{i_1,j_1,\cdots,i_m,j_m}$'s. By computing the $\mathbb{Z}_2$-intersection numbers we know these $2^{k-1}-1$ surfaces represent distinct $\mathbb{Z}_2$-homology classes in $H_2(M;\mathbb{Z}_2)$ and do not lie in the first part $\mathbb{Z}_2^{2g}$.
Since the second direct sum part $\mathbb{Z}_2^{k-1}$ also has $2^{k-1}-1$ non-zero elements, we have these $V_{i_1,j_1,\cdots,i_m,j_m}$'s represent the non-zero elements in the second direct sum part $\mathbb{Z}_2^{k-1}$.

Proof of Step 2.

Given $f \in H_2(M;\mathbb{Z}_2) = \mathbb{Z}_{2g}\oplus H$, we can write it as $f = f_1 + f_2$, where $f_1 \in \mathbb{Z}_{2g}$ and $f_2 \in H$.
$f_1$ can be represented by a vertical torus. If $f_2$ is represented by a union of pseudo-vertical surfaces, then it does not intersect the vertical torus. We can take the union to represent $f$. If $f_2$ is represented by a horizontal surface or a pseudo-horizontal surface, then it intersects the vertical torus in a simple closed curve. We take the union of two surfaces and do surgeries. The result of the surgeries is still a horizontal surface or a pseudo-horizontal surface, whose Euler characteristic remains the same.
\end{proof}

\begin{Remark}\label{rem-norm-upper-bound}
In Step 1 of Lemma \ref{lem-represent-homology-class}, the Euler characteristics of the elements in $\mathbb{Z}_2^{2g}$ are zero. In Case (1.2), the absolute value of the Euler characteristic of the horizontal or the pseudo-horizontal surfaces can be computed by Theorem \ref{thm-genus-of-pseudo-horizontal-surface}:
$$
\lambda\left(n-2+2g - \sum_{i=1}^{n-1} \frac 1{\alpha_i} - \frac 1 {\lambda_n}\right) + N(-\lambda_n \beta_n+\mu_n \alpha_n,\lambda_n\delta_n-\mu_n\gamma_n),
$$
where $\mu_n/\lambda_n = -\beta_1/\alpha_1 - \cdots - \beta_{n-1}/\alpha_{n-1}$, $\lambda_n >0$, $\alpha_n \delta_n - \beta_n \gamma_n = 1$, $\lambda$ is the least common multiple of $\alpha_1,\cdots,\alpha_{n-1},\lambda_n$. In Case (2), the non-orientable genus of $V_{i_1,j_1,\cdots,i_m,j_m}$ is no more than $\sum_{2 \mid i} N(\alpha_i,\beta_i)$. Hence we get a uniform upper bound for the $\mathbb{Z}_2$-Thurston norms for the elements in $H_2(M;\mathbb{Z}_2)$.
\end{Remark}

\begin{Remark}\label{rem-homology-ring}
As a by-product, using the intersections of the concrete closed curves and surfaces that represent the elements $H_1(M;\mathbb{Z}_2)$ and $H_2(M;\mathbb{Z}_2)$, we can easily calculate the $\mathbb{Z}_2$-cohomology ring of Seifert fibered manifolds. We can compare this to the complicated method to calculate the cup product of the cohomology ring in \cite{BHZZ}.
\end{Remark}

When there are $k>1$ even numbers among $\alpha_1,\cdots,\alpha_n$, the Seifert fibered manifold $M=(\Sigma_g,(\alpha_1,\beta_1),\cdots,(\alpha_n,\beta_n))$ can have many pseudo-horizontal surfaces representing the elements in the second direct sum part $\mathbb{Z}_2^{k-1}$ of $H_2(M;\mathbb{Z}_2)=\mathbb{Z}_2^{2g}\oplus \mathbb{Z}_2^{k-1}$. We can detect the homology classes of these pseudo-horizontal surfaces by the following lemma and the structure of the $\mathbb{Z}_2$-cohomology ring.

\begin{Lemma}\label{lem-homology-class-horizontal}
We follow the notations at the beginnings of Section 2 and Section 3 and the notations of Lemma \ref{lem-homology-group-presentation}.
The $\mathbb{Z}_2$-intersection number of $[Z]$ and the basis of $H_1(M;\mathbb{Z}_2)$ is as follows.
\begin{enumerate}
  \item If each of $\alpha_1,\cdots,\alpha_n$ is odd, then $[Z]\cdot[a_i]=0$, $[Z] \cdot [b_i]=0\;(i = 1,\cdots,g)$ and $[Z]\cdot[v]=1$.
  \item If there are $k>1$ even numbers among $\alpha_1,\cdots,\alpha_n$, then $[Z]\cdot[a_i]=0, [Z] \cdot [b_i]=0\;(i = 1,\cdots,g)$ and $[Z]\cdot[h_j]=\mu_j \cdot \lambda / \lambda_j\;(j = 1,\cdots,n)$.
\end{enumerate}
\end{Lemma}

\begin{proof}
Direct geometric construction. For a picture, see Figure 3 in Example \ref{first-example}.
\end{proof}

\section{Examples}\label{section-eg}

In this section, we first give three families of small Seifert fibered manifolds in which there exist pseudo-vertical surfaces that are not $\mathbb{Z}_2$-taut. Then in the next example, a pseudo-vertical surface is isotopic to a pseudo-horizontal surface. In the final example, the three pseudo-vertical surfaces are $\mathbb{Z}_2$-taut and all the pseudo-horizontal surfaces are not $\mathbb{Z}_2$-taut. By convention, we denote the 2-sphere $\Sigma_0$ by $S^2$. Throughout this section, we follow the notations at the beginnings of Section 2 and Section 3 and the notations of Definition \ref{def-pseudo}.

\begin{Example}\label{first-example}
$M = (S^2,(2,-1),(2m+1,m),(2n,1))$, where $n > 2m+1$.
\end{Example}
In this example,
$$
\left(\begin{matrix}\alpha_1 & \beta_1\\ \gamma_1 & \delta_1\end{matrix}\right) = \left(\begin{matrix}2 & -1\\-1 & 1\end{matrix}\right),
\left(\begin{matrix}\alpha_2 & \beta_2\\ \gamma_2 & \delta_2\end{matrix}\right) = \left(\begin{matrix}2m+1 & m\\2 & 1\end{matrix}\right),
\left(\begin{matrix}\alpha_3 & \beta_3 \\ \gamma_3 & \delta_3\end{matrix}\right) = \left(\begin{matrix}2n & 1\\ 2n-1 & 1\end{matrix}\right).
$$

Take $(\lambda_1,\mu_1) = (2,-1)$,
$(\lambda_2,\mu_2) = (2m+1,m)$,
and $(\lambda_3,\mu_3) = (4m+2,1)$.
The least common multiple $\lambda = 4m+2$.
By Theorem \ref{thm-existence-for-pseudo-horizontal-surface} $(\lambda_1,\mu_1),(\lambda_2,\mu_2),(\lambda_3,\mu_3)$ determine a pseudo-horizontal surface $Z$.
By Theorem \ref{thm-genus-of-pseudo-horizontal-surface}, the non-orientable genus of $Z$ is
\begin{eqnarray*}
& & 2 + (4m+2) \left( 1 - \frac{1}{2} - \frac{1}{2m+1} - \frac{1}{4m+2} \right) \\
& & \quad + N(-(4m+2) + 2n,(4m+2) - (2n-1)) \\
&=& 2m + N(2n-(4m+2),1) = 2m + n - (2m+1) = n - 1.
\end{eqnarray*}
By Lemma \ref{lem-homology-class-vertical} and Lemma \ref{lem-homology-class-horizontal}, $Z$ is in the same $\mathbb{Z}_2$-homology class as the pseudo-vertical surface $[V_{1,3}]$. By Lemma \ref{lem-genus-pseudo-vertical}, the non-orientable genus of $V_{1,3}$ is $n+1$.

Figure 2 shows the identification from each $\partial R_i$ to $\partial M_0$ for the small Seifert manifold $M = (S^2,(2,-1),(3,1),(8,1))$. The upper side should be identified to the bottom side. Figure 2-(left) shows the result of deleting a tubular neighborhood of a regular fiber from $M_0$, which is denoted by $M_0'$. Figure 2-(right) shows three solid tori. We draw the corresponding curves for the identification.
\begin{figure}[htbp]
\begin{center}
\includegraphics[height=6cm]{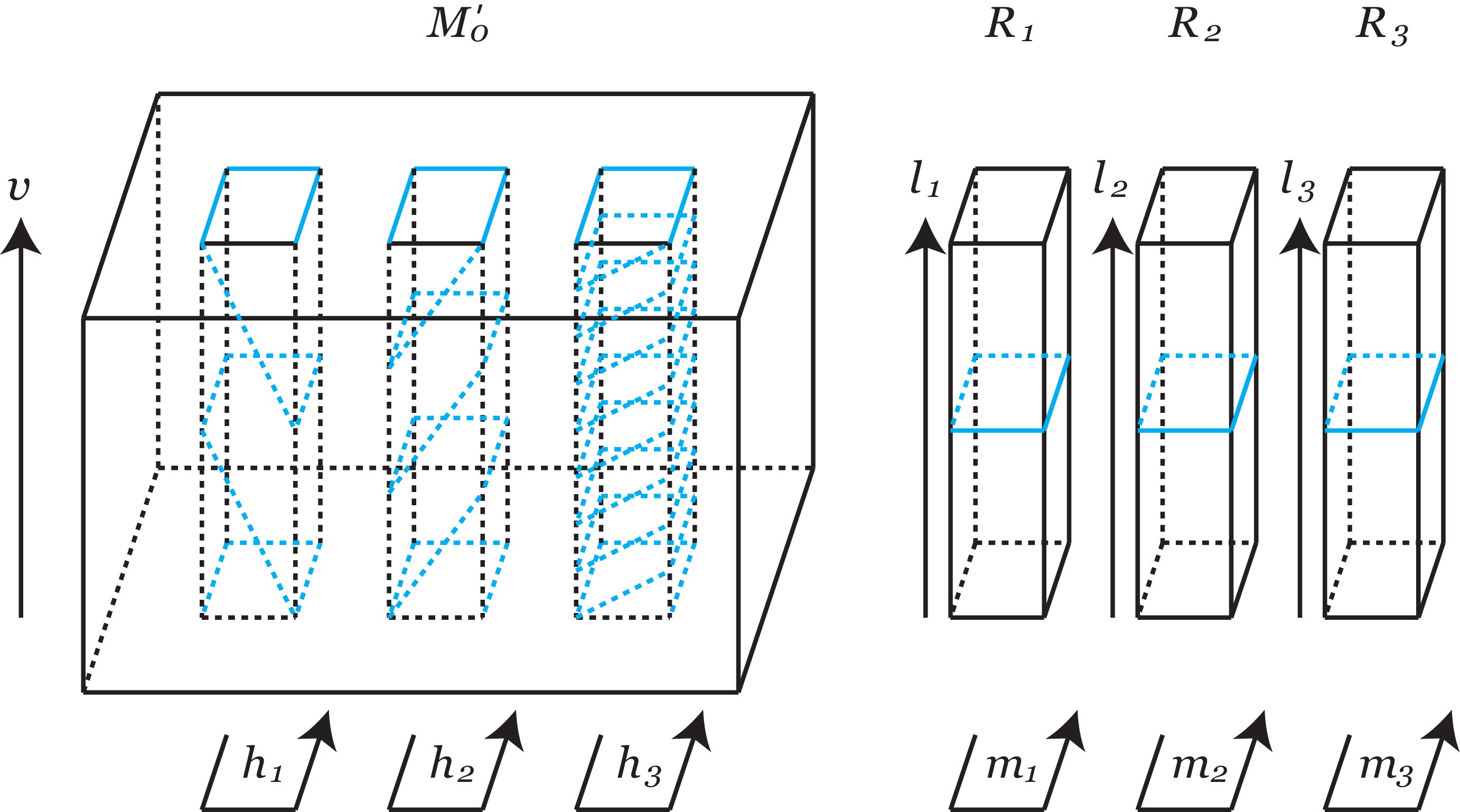}\\
Figure 2.
\end{center}
\end{figure}

Figure 3 shows the pseudo-horizontal surface $Z$ determined by $(\lambda_1,\mu_1)=(2,-1)$, $(\lambda_2,\mu_2)=(3,1)$, $(\lambda_3,\mu_3)=(6,1)$ inside $(S^2,(2,-1),(3,1),(8,1))$.
Figure 3-(left) is the front part of $M_0'$. Figure 3-(middle) is the back part of $M_0'$. Figure 3-(right) shows five disks and one M\"obius band in the solid tori $R_1,R_2,R_3$.
We can see the staircase surface in $M_0'$. We can check $[Z] = [V_{1,3}]$, where $V_{1,3}$ is the pseudo-vertical surface connecting $R_1$ and $R_3$.
The non-orientable genus of $Z$ is 3, while the non-orientable genus of $V_{1,3}$ is $N(2,1)+N(8,1)=5$.

\begin{figure}[htbp]
\begin{center}
\includegraphics[height=4.3cm]{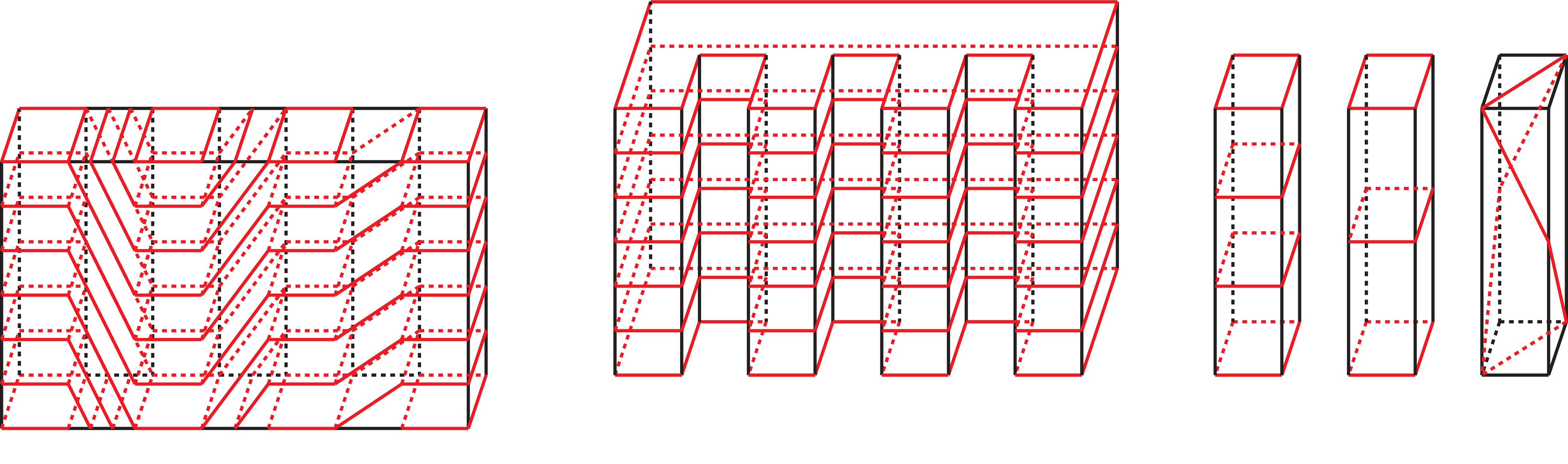}\\
Figure 3.
\end{center}
\end{figure}

\begin{Example}
$M = (S^2,(3,-1),(4,1),(2n,1))$, where $n > 6$.
\end{Example}

In this example,
$$
\left(\begin{matrix}\alpha_1 & \beta_1\\ \gamma_1 & \delta_1\end{matrix}\right) = \left(\begin{matrix}3 & -1\\-2 & 1\end{matrix}\right),
\left(\begin{matrix}\alpha_2 & \beta_2\\ \gamma_2 & \delta_2\end{matrix}\right) = \left(\begin{matrix}4 & 1\\3 & 1\end{matrix}\right),
\left(\begin{matrix}\alpha_3 & \beta_3 \\ \gamma_3 & \delta_3\end{matrix}\right) = \left(\begin{matrix}2n & 1\\ 2n-1 & 1\end{matrix}\right).
$$
Let $(\lambda_1,\mu_1) = (3,-1)$,
$(\lambda_2,\mu_2) = (4,1)$,
and $(\lambda_3,\mu_3) =  (12,1)$.
Then $\lambda = 12$ and $(\lambda_1,\mu_1),(\lambda_2,\mu_2),(\lambda_3,\mu_3)$ determine a pseudo-horizontal surface $Z$ whose non-orientable genus is
$$
2 + 12 \left( 1 - \frac{1}{3} - \frac{1}{4} - \frac{1}{12}\right) + N(-12 + 2n,12 - 2n + 1) =  6 + N(2n-12,1) = n.
$$
We can check $[Z] = [V_{2,3}]$, where $V_{2,3}$ is the pseudo-vertical surface connecting $R_2$ and $R_3$.
The non-orientable genus of $V_{2,3}$ is $n+2$.

\begin{Example}
$M = (S^2,(m,-1),(2n_2,1),(2n_3,1))$ where $m \geq 2$, $n_2 \geq m$, $n_3 \geq m$, and $n_2 + n_3 > 2m$.
\end{Example}

In this example,
$$
\left(\begin{matrix}\alpha_1 & \beta_1\\ \gamma_1 & \delta_1\end{matrix}\right) = \left(\begin{matrix}m & -1\\1-m & 1\end{matrix}\right),
\left(\begin{matrix}\alpha_2 & \beta_2\\ \gamma_2 & \delta_2\end{matrix}\right) = \left(\begin{matrix}2n_2 & 1\\2n_2-1 & 1\end{matrix}\right),
\left(\begin{matrix}\alpha_3 & \beta_3 \\ \gamma_3 & \delta_3\end{matrix}\right) = \left(\begin{matrix}2n_3 & 1\\ 2n_3-1 & 1\end{matrix}\right).
$$
Let $(\lambda_1,\mu_1) = (m,-1)$,
$(\lambda_2,\mu_2) = (2m,1)$,
$(\lambda_3,\mu_3) = (2m,1)$.
Then $\lambda = 2m$ and $(\lambda_1,\mu_1),(\lambda_2,\mu_2),(\lambda_3,\mu_3)$ determine a pseudo-horizontal surface $Z$ whose non-orientable genus is
$$
2 + 2m\left(1 - \frac{1}{m} - \frac{1}{2m} - \frac{1}{2m}\right) + N(2n_2 - 2m,1) + N(2n_3-2m,1) = n_2 + n_3 - 2.
$$
Such a pseudo-horizontal surface intersects $h_2$ and $h_3$ once respectively. Hence $[Z] = [V_{2,3}]$, where $V_{2,3}$ is the pseudo-vertical surface connecting $R_2$ and $R_3$. The non-orientable genus of $[V_{2,3}]$ is $n_2 + n_3$.

\begin{Example}\label{eg-O48}
$M = (S^2,(2,-1),(3,1),(4,1))$. This manifold is the elliptic 3-manifold $S^3/O_{48}^*$ where $O_{48}^*$ is the binary octahedral group.
\end{Example}

In this example,
$$
\left(\begin{matrix}\alpha_1 & \beta_1\\ \gamma_1 & \delta_1\end{matrix}\right) = \left(\begin{matrix}2 & -1\\-1 & 1\end{matrix}\right),
\left(\begin{matrix}\alpha_2 & \beta_2\\ \gamma_2 & \delta_2\end{matrix}\right) = \left(\begin{matrix}3 & 1\\2 & 1\end{matrix}\right),
\left(\begin{matrix}\alpha_3 & \beta_3 \\ \gamma_3 & \delta_3\end{matrix}\right) = \left(\begin{matrix}4 & 1\\ 3 & 1\end{matrix}\right).
$$
Let $(\lambda_1,\mu_1) = (2,-1)$,
$(\lambda_2,\mu_2) = (3,1)$,
$(\lambda_3,\mu_3) = (6,1)$.
They determine a pseudo-horizontal surface $Z$ whose non-orientable genus is
$$
2 + 6\left( 1 - \frac{1}{2} - \frac{1}{3} - \frac{1}{6} \right)
+N(-6 \cdot 1 + 1\cdot 4,6 \cdot 1 - 1 \cdot 3)
= 2 + N(2,1) = 3.
$$
The non-orientable genus of the pseudo-vertical surface $V_{1,3}$ is also $3$.
By \cite{RB}, the mapping class group of $M$ is trivial.
Hence $[V_{1,3}]$ is isotopic to $[Z]$.

\begin{Example}\label{eg-prism}
$M = (S^2,(2,-1),(2,1),(2n,1))$, where $n \geq 1$. These manifolds are also called ``prism manifolds'' which are elliptic 3-manifolds. Their fundamental groups are called ``binary dihedral groups'' or ``generalized quaternion groups'', denoted by $D^*_{8n}$ or $Q_{8n}$.
\end{Example}

In these examples,
$$
\left(\begin{matrix}\alpha_1 & \beta_1\\ \gamma_1 & \delta_1\end{matrix}\right) = \left(\begin{matrix}2 & -1\\-1 & 1\end{matrix}\right),
\left(\begin{matrix}\alpha_2 & \beta_2\\ \gamma_2 & \delta_2\end{matrix}\right) = \left(\begin{matrix}2 & 1\\1 & 1\end{matrix}\right),
\left(\begin{matrix}\alpha_3 & \beta_3 \\ \gamma_3 & \delta_3\end{matrix}\right) = \left(\begin{matrix}2n & 1\\ 2n-1 & 1\end{matrix}\right).
$$

There are three pseudo-vertical surfaces $V_{1,2},V_{1,3},V_{2,3}$. These surfaces represent the non-zero elements in $H_2(M;\mathbb{Z}_2) = \mathbb{Z}_2 \oplus \mathbb{Z}_2$. We can use the algorithm in Section 7 to prove they are $\mathbb{Z}_2$-taut. This result has also been shown in \cite[Section 6]{JRT} by a different method. Moreover, the algorithm in Section 7 will tell us all pseudo-horizontal surfaces in $M$ are not $\mathbb{Z}_2$-taut.

\section{Properties of the $N(2k,q)$}

The function $N(2k,q)$ plays a crucial role in the calculation of the $\mathbb{Z}_2$-Thurston norms for general orientable Seifert manifolds. This section studies more properties of $N(2k,q)$ when $k$ and $q$ are large.

For a general fractal number $\alpha/\gamma$ ($\alpha>\gamma>0$, $(\alpha,\gamma)=1$, $\alpha$ needs not to be even), write it as the continued fraction $[a_0,a_1,\cdots,a_n]$. Just as the laws in Proposition \ref{prop-value-of-N}, construct a new number sequence $(b_0,b_1,\cdots,b_n)$, i.e. $b_0=a_0$,
$$
  b_i = \left\{ \begin{array}{ll}a_i, & b_{i-1} = 0 \text{ or } \sum_{j=0}^{i-1} b_j \text{ is odd}, \\
  0, & b_{i-1} \neq 0 \text{ and } \sum_{j=0}^{i-1} b_j \text{ is even}. \end{array}\right.
$$

\begin{Lemma}\label{lem-subsequence-cases}
Let positive integers $\alpha,\gamma$ and sequence $(b_i)_{i=1}^n$ be as above. Suppose there are $k$ zero's in $(b_i)$, these zero terms are $b_{m_1},\cdots,b_{m_k}$. They separate $(b_i)$ into many subsequences, and the terms in the subsequences equal the corresponding terms in $(a_i)$. Exactly one of the following 3 cases will happen for each subsequence.
\begin{enumerate}
  \item[(1)] The subsequence has only 1 term whose number is even.
  \item[(2)] The subsequence has at least 2 terms. The numbers of the first term and last term are odd. The numbers of the middle terms are even.
  \item[(3)] The number of the first term of the subsequence is odd. The numbers of the other terms are even.
\end{enumerate}
Case (3) can only happen for the last subsequence $(b_{m_k+1},b_{m_k+2},\cdots,b_n)$.
\end{Lemma}
\begin{proof}
The first subsequence is $b_0,\cdots,b_{m_1-1}$.
The subsequences in the middle are $b_{m_i+1},b_{m_i+2},\cdots,b_{m_{i+1}-1}\;(i=1,\cdots,k-1)$.
By the construction of $(b_i)$, we can directly verify the statements of the lemma.
\end{proof}

\begin{Lemma}\label{lem-continued-fraction-pairity}
Let positive integers $\alpha,\gamma$, sequence $(b_i)_{i=1}^n$ be as above. If $\alpha$ is even, then $\sum_{i=0}^n b_i$ is even; If $\alpha$ is odd, then $\sum_{i=0}^n b_i$ is odd.
\end{Lemma}
\begin{proof}
Suppose $\alpha/\gamma = [a_0,a_1,\cdots,a_n]$. We apply the Euclidean algorithm to $(\alpha,\gamma)$.
Take $r_0 = \alpha, r_1 = \gamma$. Then we have
$$
\begin{array}{rclrcl}
r_0 &=& r_1 \cdot a_0 + r_2, &  \qquad (r_0,r_1) & \overset{a_0}\longrightarrow & \; (r_1,r_2),\\
r_1  &=& r_2 \cdot a_1 + r_3, & \qquad (r_1,r_2) & \overset{a_1}\longrightarrow & \; (r_2,r_3),\\
r_2 &=& r_3 \cdot a_2 + r_4, & \qquad (r_2,r_3) & \overset{a_2}\longrightarrow & \; (r_3,r_4),\\
&\cdots& & &\cdots& \\
r_{n-2} &=& r_{n-1} \cdot a_{n-2} + r_n, & \qquad (r_{n-2},r_{n-1}) & \overset{a_{n-2}}\longrightarrow & \;(r_{n-1},r_n),\\
r_{n-1} &=& r_n \cdot a_{n-1} + 1, & \qquad (r_{n-1},r_n) & \overset{a_{n-1}}\longrightarrow & \;(r_n,1),\\
r_{n} &=& 1 \cdot a_n + 0, & \qquad (r_{n},1) & \overset{a_n}\longrightarrow & \; (1,0).
\end{array}
$$
We find the zero terms $b_{m_1},b_{m_2},\cdots,b_{m_k}$. By summing up from $r_0 = r_1 \cdot a_0 + r_2$ to $r_{m_1-1} = r_{m_1}\cdot a_{m_1-1}+r_{m_1+1}$ and modulo 2, we get $2 \mid (r_0-r_{m_1+1})$. Since $(b_0,\cdots,b_{m_1-1})=(a_0,\cdots,a_{m_1-1})$ and $b_{m_1}=0$, when we calculate $b_{m_1+1}$, we need to begin with $r_{m_1+1} = r_{m_1+2} \cdot a_{m_1+1} + r_{m_1+3}$. By summing up from $r_{m_1+1} = r_{m_1+2} \cdot a_{m_1+1} + r_{m_1+3}$ to $r_{m_2-1} = r_{m_2}\cdot a_{m_2-1}+r_{m_2+1}$ and modulo 2, we get $2 \mid (r_{m_1+1}-r_{m_2+1})$. Similarly, we finally get $2 \mid (r_0 - r_{m_k+1})$.

Again, sum up from $r_{m_k+1} = r_{m_k+2} \cdot a_{m_k+1} + r_{m_k+3}$ to $r_{n} = 1\cdot a_{n}+0$ and modulo 2.
If the last subsequence $(a_{m_k+1},a_{m_k+2},\cdots,a_n)=(b_{m_k+1},b_{m_k+2},\cdots,b_n)$ belongs to case (3) in Lemma \ref{lem-subsequence-cases}, then $\sum_{i=0}^n b_i$ is odd and $r_{m_k+1}$ is odd, hence $\alpha$ is odd. If the last subsequence $(a_{m_k+1},a_{m_k+2},\cdots,a_n)=(b_{m_k+1},b_{m_k+2},\cdots,b_n)$ belongs to case (1) or (2) in Lemma \ref{lem-subsequence-cases}, then $\sum_{i=0}^n b_i$ is even and $r_{m_k+1}$ is even, hence $\alpha$ is even.
\end{proof}

\begin{Lemma}\label{lem-asymptotic-N}
Suppose $\alpha,\gamma,x,y$ are fixed integers, $(\alpha,\gamma)=1$, $0<\gamma<\alpha$, and $N(\cdot,\cdot)$ is the function described in Proposition \ref{prop-value-of-N}. Let $L=|\alpha|+|\gamma|+|x|+|y|$. When (i) $\mu > 8L^{3}$, (ii) $(\alpha \cdot \mu + x,\gamma \cdot \mu + y)=1$ and (iii) $\alpha \cdot \mu + x$ is even, we have the following properties for $ N(\alpha \cdot \mu + x,\gamma \cdot \mu + y)$:
\begin{itemize}
  \item[(1)] If $\alpha$ is even, then $N(\alpha \cdot \mu + x,\gamma \cdot \mu + y) \geq N(\alpha,\gamma)$. Moreover, there exists a positive integer $t < 8L^{3}$ such that $N(\alpha \cdot (\mu + t) + x,\gamma \cdot (\mu + t) + y) = N(\alpha \cdot \mu + x,\gamma \cdot \mu + y)$.
  \item[(2)] If $\alpha$ is odd, then $N(\alpha \cdot \mu + x,\gamma \cdot \mu + y) \to +\infty$ as $\mu \to +\infty$. More specifically, for any $G>0$, if $\mu > (G + 2)\cdot 8L^{3}$, then $N(\alpha\cdot \mu + x,\gamma \cdot \mu +y) > G$.
\end{itemize}
\end{Lemma}

\begin{proof}
Write down the continued fraction for $\alpha/\gamma$ and $(\alpha \cdot \mu + x)/(\gamma \cdot \mu + y)$ and compare them.
The continued fraction of $\alpha/\gamma$ is given in Lemma \ref{lem-continued-fraction-pairity}, i.e. take $r_0=\alpha, r_1 = \gamma$, then $r_i = r_{i+1} \cdot a_i + r_{i+2}$, with $r_{n+1}=1,r_{n+2}=0$. The $a_i,r_i\;(i=0,\cdots,n)$ are positive integers determined by $\alpha$ and $\gamma$. We can see $a_i < L, r_i<L, n<L, r_{i+1}<r_i$.

When $\mu> 8L^{3}$, we now show the first $n$ terms of the continued fraction of $\alpha/\gamma$ and $(\alpha \cdot \mu + x)/(\gamma \cdot \mu + y)$ are the same. Let $z_0=x,z_1=y$ and $z_{i+2} = z_i - z_{i+1}\cdot a_i\;(i=0,1,\cdots,n)$. Then formally we have
$$
\begin{array}{rcl}
r_0 \cdot \mu + z_0 &=& (r_1 \cdot \mu + z_1) \cdot a_0 + (r_2 \cdot \mu + z_2),\\
r_1 \cdot \mu + z_1 &=& (r_2 \cdot \mu + z_2) \cdot a_1 + (r_3 \cdot \mu + z_3),\\
r_2 \cdot \mu + z_2 &=& (r_3 \cdot \mu + z_3) \cdot a_2 + (r_4 \cdot \mu + z_4),\\
\cdots\\
r_{n-2} \cdot \mu + z_{n-2} &=& (r_{n-1} \cdot \mu + z_{n-1}) \cdot a_{n-2} + (r_{n} \cdot \mu + z_{n}),\\
r_{n-1} \cdot \mu + z_{n-1} &=& (r_n \cdot \mu + z_n) \cdot a_{n-1} + (1 \cdot \mu + z_{n+1}),\\
r_n \cdot \mu + z_n &=& (1 \cdot \mu + z_{n+1}) \cdot a_n +  z_{n+2}.
\end{array}
$$
We estimate $|z_i|$ as follows. By the extended Euclidean algorithm, for $2 \leq i \leq n+1$, we can compute $u_i,v_i$ satisfying $|u_i|<L, |v_i|<L$ such that $\alpha \cdot u_i + \gamma \cdot v_i = r_i$ and $x \cdot u_i + y \cdot v_i = z_i$. Hence for $2 \leq i \leq n+1$ we have $|z_i| < 2L^2$. For $z_{n+2}$, we have $|z_{n+2}| \leq |z_n| + |z_{n+1}\cdot a_n| < 4L^3$.

If both (a) $r_i \cdot \mu + z_i>0\;(i=1,\cdots,n)$ and (b) $r_i \cdot \mu + z_i > r_{i+1} \cdot \mu + z_{i+1}\;(i = 1,\cdots,n)$ hold, then the above process will give the first $n$ terms of the continued fraction for $(\alpha \cdot \mu + x)/(\gamma \cdot \mu +y)$.

Check (a): we only need $r_i \cdot \mu > |z_i|$. In fact, $|z_i| < 2L^2 < \mu < r_i \cdot \mu$.

Check (b): we only need $(r_i-r_{i+1})\cdot \mu > |z_{i+1}-z_i|$. In fact, $|z_{i+1}-z_i|< 4L^2 < \mu < (r_i-r_{i+1})\cdot \mu$.

We look at the rest steps of the Euclidean algorithm for $(\alpha \cdot \mu + x,\gamma \cdot \mu + y)$. In the $(n+1)$-th step $r_n \cdot \mu + z_n = (1 \cdot \mu + z_{n+1}) \cdot a_n +  z_{n+2}$, we do not know whether $z_{n+2}>0$ or $z_{n+2}<0$. So we need to discuss both cases.

If $z_{n+2}>0$, then we have
$$
\begin{array}{rcl}
r_{n} \cdot \mu + z_n &=& (1 \cdot \mu + z_{n+1})\cdot a_n + z_{n+2},\\
1 \cdot \mu + z_{n+1} &=& z_{n+2} \cdot \left\lfloor \frac{\mu + z_{n+1}}{z_{n+2}} \right\rfloor + s_{n+3},\\
r_{n+2} \cdot \mu + z_{n+2} &=& s_{n+3} \cdot a_{n+2} + s_{n+4},\\
&\cdots&
\end{array}
$$
where $s_{n+3},s_{n+4},a_{n+2}>0$, $\lfloor \cdot \rfloor$ is the floor function. So the continued fraction expansion for $(\alpha \cdot \mu + x)/(\gamma \cdot \mu + y)$ is
$$
(\alpha\cdot \mu + x)/(\gamma \cdot \mu + y) = \left[a_0,a_1,\cdots,a_{n-1},a_n,\left\lfloor \cfrac{\mu + z_{n+1}}{z_{n+2}} \right\rfloor,a_{n+2},\cdots\right]
$$
If $z_{n+2}<0$, in order to make the numbers positive, we need to modify a little:
$$
\begin{array}{rcl}
r_{n} \cdot \mu + z_n &=& (1 \cdot \mu + z_{n+1})\cdot (a_n-1) + [1 \cdot \mu + (z_{n+1}+z_{n+2})],\\
1 \cdot \mu + z_{n+1} &=& [1 \cdot \mu + (z_{n+1}+z_{n+2})] \cdot 1 + (-z_{n+2}),\\
1 \cdot \mu + (z_{n+1}+z_{n+2}) &=& (-z_{n+2}) \cdot \left\lfloor \frac {1\cdot \mu + (z_{n+1}+z_{n+2})}{-z_{n+2}} \right\rfloor + s_{n+4},\\
-z_{n+2} &=& s_{n+4}\cdot a_{n+3} + s_{n+5},\\
&\cdots&
\end{array}
$$
So the continued fraction expansion for $(\alpha \cdot \mu + x)/(\gamma \cdot \mu + y)$ is
$$
\left[ a_0,a_1,\cdots,a_{n-1},a_n-1,1,\left\lfloor \frac{1 \cdot \mu + z_{n+1} + z_{n+2}}{(-z_{n+2})} \right\rfloor,a_{n+3},\cdots \right]
$$

We construct the new sequences $(b_i)$ and $(b_i')$ for calculating $N(\alpha,\gamma)$ and $N(\alpha\cdot\mu+x,\gamma\cdot \mu +y)$ according to the law at the beginning of this section.

When $\alpha$ is even, $\sum_{i=0}^n b_i$ is even. If $z_{n+2}>0$, then $\sum_{i=0}^n b_i'$ is also even. It makes $b_{n+1}'=0$, the term $\lfloor \cfrac{\mu + z_{n+1}}{z_{n+2}} \rfloor$ does not contribute to $ N(\alpha \cdot \mu + x,\gamma \cdot \mu + y)$. If $z_{n+2}<0$, then $\sum_{i=0}^n b_i' = \sum_{i=0}^n b_i-1$ is odd. Now $b_{n+1}'=1$, so $\sum_{i=0}^{n+1} b_i'$ is even. It makes $b_{n+2}'=0$, the term $\lfloor \cfrac{1 \cdot \mu + z_{n+1} + z_{n+2}}{(-z_{n+2})} \rfloor$ does not contribute to $ N(\alpha \cdot \mu + x,\gamma \cdot \mu + y)$. In each case, taking $\mu' = \mu + |z_{n+2}|$, the rest part of the continued fraction expansions of $(\alpha \cdot \mu + x)/(\gamma \cdot \mu + y)$ and $(\alpha \cdot \mu' + x)/(\gamma \cdot \mu' + y)$ are the same. Let $t = |z_{n+2}|$, then we have
$$N(\alpha \cdot (\mu+t)+x,\gamma \cdot (\mu+t) + y) = N(\alpha \cdot \mu + x, \gamma\cdot \mu + y).$$

When $\alpha$ is odd, $\sum_{i=0}^n b_i$ is odd. If $z_{n+2}>0$, then $\sum_{i=0}^n b_i' = \sum_{i=0}^n b_i$ is odd. When we compute $N(\alpha \cdot \mu+x,\gamma \cdot \mu + y)$ we must add $b_{n+1}' = \lfloor \cfrac{\mu + z_{n+1}}{z_{n+2}} \rfloor$. If $z_{n+2}<0$, then $\sum_{i=0}^n b_i' = \sum_{i=0}^n b_i-1$ is even and $b_{n+1}'=0$. When we compute $N(\alpha \cdot \mu+x,\gamma \cdot \mu + y)$ we must add $b_{n+2}' = \lfloor \cfrac{1 \cdot \mu + z_{n+1} + z_{n+2}}{(-z_{n+2})} \rfloor$. Since $|z_{n+2}|<4L^{3}$ and $|z_{n+1}|<2L^2$, no matter $z_{n+2}>0$ or $z_{n+2}<0$, for arbitrarily large $G$, when $\mu > (G +2)\cdot 8L^3$, we have $N(\alpha \cdot \mu + x, \gamma \cdot \mu + y) > G$.
\end{proof}

\begin{Remark}\label{rem-asymptotic-N-negative}
Lemma \ref{lem-asymptotic-N} only deals with the $\mu > 0$ case. If $\mu < 0$, then $N(\alpha \cdot \mu + x, \gamma \cdot \mu + y) = N(\alpha \cdot (-\mu) - x, \gamma \cdot (-\mu) - y)$. Hence we can reduce the $\mu < 0$ case to the $\mu > 0$ case.
\end{Remark}

\section{$\mathbb{Z}_2$-Thurston norms for Seifert manifolds}

Suppose we have a Seifert fibered manifold  $M=(\Sigma_g,(\alpha_1,\beta_1),\cdots,(\alpha_n,\beta_n))$. For each non-zero element in $H_2(M;\mathbb{Z}_2)$, we calculate its $\mathbb{Z}_2$-Thurston norm as follows.

\textbf{Step 1.} In Remark \ref{rem-norm-upper-bound} we have described a uniform upper bound for the $\mathbb{Z}_2$-Thurston norms of the elements in $H_2(M;\mathbb{Z}_2)$. Denote this upper bound by $g_{\max}-2$.

\textbf{Step 2.} Take all pseudo-vertical surfaces and their combinations as candidates. As in Case (2) of Step 1 of the proof of Lemma \ref{lem-represent-homology-class}, there are $2^{k-1}-1$ of them. Here $k$ is the cardinality of the even numbers in $\{\alpha_1,\cdots,\alpha_n\}$.

Take the following finite cases of $(\lambda_1,\mu_1),\cdots,(\lambda_n,\mu_n)$ for the horizontal surfaces or pseudo-horizontal surfaces.
\begin{enumerate}
  \item $\lambda_1=\cdots=\lambda_n=1$ and $\mu_1=\cdots=\mu_n=0$.
  \item $(\lambda_j,\mu_j) = (\alpha_j,\beta_j)\;(1 \leq j \leq n)$.
  \item for $1 \leq i\leq n$, take $(\lambda_j,\mu_j) = (\alpha_j,\beta_j)$ ($j \neq i$) and $\mu_i/\lambda_i = -\sum_{j \neq i} \beta_j/\alpha_j$.
\end{enumerate}
Check whether they satisfy the condition in Theorem \ref{thm-existence-for-pseudo-horizontal-surface}. If they satisfy, then build a horizontal surface or pseudo-horizontal surface as a candidate surface. In the following steps we exclude the above cases.

\textbf{Step 3.} Take the pseudo-horizontal surface $Z$ determined by $(\lambda_i,\mu_i)\;(i=1,\cdots,n)$. Denote its non-orientable genus by $g^N(Z)$. We will prove if the least common multiple $\lambda$ is larger than $2 g_{\max} + 2$, then $g^N(Z) > g_{\max}$ and hence $Z$ is not $\mathbb{Z}_2$-taut. So the candidate pseudo-horizontal surfaces satisfy $\lambda_i \leq 2 g_{\max} + 2\;(i=1,\cdots,n)$.

\textbf{Step 4.} Suppose $\lambda_i \leq 2 g_{\max}+2\;(i=1,\cdots,n)$ are fixed. We will prove that the candidate $\mathbb{Z}_2$-taut pseudo-horizontal surfaces must satisfy
$$
|\mu_i| \leq (8 \cdot g_{\max}+16)\cdot [(4g_{\max}+6)\cdot K]^3
$$
for each $i$. Here $K = \max_{1 \leq i \leq n}(|\alpha_i| + |\beta_i|)$.

\textbf{Step 5.} Now we have finitely many candidate pseudo-horizontal surfaces and pseudo-vertical surfaces.
We can calculate their $\mathbb{Z}_2$-homology classes as in Section 4. So for every given $\mathbb{Z}_2$-homology class, we can find the minimal genus of the representing surfaces.

The only remaining thing is to prove Step 3 and Step 4.

\begin{proof}[Proof of Step 3]
If $(\alpha_i,\beta_i),(\lambda_i,\mu_i)\;(i=1,\cdots,n)$ satisfy the condition of Theorem \ref{thm-existence-for-pseudo-horizontal-surface}, then by Theorem \ref{thm-genus-of-pseudo-horizontal-surface}, we have
$$
g^N(Z) = 2 + \lambda\left(n-2+2g - \sum_{i=1}^n \frac 1{\lambda_i}\right) + \sum_{i=1}^n N(-\lambda_i \beta_i+\mu_i \alpha_i,\lambda_i\delta_i-\mu_i\gamma_i).
$$

If $\max(\lambda_1,\cdots,\lambda_n) \neq \lambda$, then $\lambda_i<\lambda\;(1\leq i \leq n)$. By Proposition \ref{prop-lcm-restriction}, $(\lambda_i,\beta_i) = (\alpha_i,\beta_i)\;(i=1,\cdots,n)$. This case is excluded by Step 2.

If there is only one of $\lambda_1, \cdots,\lambda_n$ equals $\lambda$. Assume $\lambda_i = \lambda$. By Property \ref{prop-parameters-identity} and Property \ref{prop-lcm-restriction}, we have $(\lambda_j,\mu_j) = (\alpha_j,\beta_j)\;(j \neq i)$ and $\mu_i/\lambda_i = - \sum_{j \neq i} \beta_j/\alpha_j$. This case is excluded by Step 2.

It remains that $\lambda>1$ and there exists $h>1$ and $k_1,k_2,\cdots,k_h \in \{1,\cdots,n\}$ such that $\lambda_{k_1},\lambda_{k_2},\cdots,\lambda_{k_h} = \lambda$. By Property \ref{prop-lcm-restriction}, for $i \not \in \{k_1,\cdots,k_h\}$ we have $(\lambda_i,\mu_i) = (\alpha_i,\beta_i)$. By Remark \ref{rem-Seifert-notation}, we have $\lambda_i \geq 2$. Then
$$
\begin{array}{rcl}
g^N(Z) &\geqslant& 2 + \lambda\left(n-2+2g - \sum_{i \in \{k_1,\cdots,k_h\}} \frac 1{\lambda_i} - \sum_{i \not\in \{k_1,\cdots,k_h\}} \frac 1{\lambda_i}\right)\\
&\geqslant& 2 + \lambda\cdot \left(n - 2 + 2g - \frac h{\lambda} - \frac{n-h}2\right)\\
&=&  2 + h \cdot\left(\frac \lambda 2 - 1\right) + \lambda \cdot \left(\frac n2  + 2g -2\right)
\end{array}
$$
We deal with $(g,\lambda)$ case by case.

Case 1, $g \geq 1$. When $\lambda \geq 2 g_{\max}$, we have $g^N(Z) \geqslant  2 + \lambda \cdot \frac n 2> g_{\max}$.

Case 2, $g = 0$ and $n\geq 4$. Now $g^N(Z) >  \lambda \cdot \left(\frac{n+h-4}2\right) - h > \frac h2 \cdot (\lambda -2) \geqslant \lambda-2$. When $\lambda >  g_{\max} + 2$ we have $g^N(Z) > g_{\max}$.

Case 3, $g = 0$ and $n = 3$. Now $h = 2$ or $h = 3$. If $h=2$, then $g^N(Z) \geqslant \frac \lambda 2$. When $\lambda > 2 g_{\max}$ we have $g^N(Z) > g_{\max}$. If $h = 3$, then $g^N(Z) \geqslant \lambda-1$. When $\lambda > g_{\max} + 1$ we have $g^N(Z) > g_{\max}$.

Case 4, $g=0$, $n \leq 2$. This case is the lens space and is known by \cite{BW}.
\end{proof}

\begin{proof}[Proof of Step 4]
According to $\alpha_i\delta_i - \beta_i\gamma_i=1$, by the extended Euclidean algorithm, we have $|\gamma_i|\leq K$, $|\delta_i| \leq K$, where $K = \max_{1 \leq i \leq n}(|\alpha_i|+|\beta_i|)$. By Theorem \ref{thm-genus-of-pseudo-horizontal-surface}, the non-orientable genus of the pseudo-horizontal surface is larger than $N(-\lambda_i \beta_i+\mu_i \alpha_i,\lambda_i\delta_i-\mu_i\gamma_i)$.

We apply Lemma \ref{lem-asymptotic-N} and Remark \ref{rem-asymptotic-N-negative}. Since $\lambda_i \leq 2g_{\max}+2$, the number $L$ in the lemma becomes $L = |\alpha_i|+|\gamma_i|+|\lambda_i\beta_i|+|\lambda_i\delta_i| \leq (4 g_{\max}+6)\cdot K$.

When $\alpha_i$ is even and $\mu_i > 8L^3$, there exist $t<8L^3$ such that $N(-\lambda_i \beta_i+\mu_i \alpha_i,\lambda_i\delta_i-\mu_i\gamma_i) = N(-\lambda_i \beta_i+(\mu_i+t) \alpha_i,\lambda_i\delta_i-(\mu_i+t)\gamma_i)$. We only need to try for $|\mu_i| < 16 L^3 = 16 \cdot [(4g_{\max}+6)\cdot K]^{3}$.

When $\alpha_i$ is odd and $|\mu_i|>(g_{\max} + 2) \cdot 8\cdot[(4g_{\max}+6)\cdot K]^{3}$, we have $N(-\lambda_i \beta_i+\mu_i \alpha_i,\lambda_i\delta_i-\mu_i\gamma_i) > g_{\max}$.
\end{proof}

\begin{Remark}
Our algorithm proposes an exact computation of the $\mathbb{Z}_2$-Thurston norms. Hence it is stronger than the method using Heegaard-Floer homology that only gives a lower bound. We have done some computation to compare the lower bounds to the exact values. Our algorithm shows the pseudo-horizontal surfaces presented in Example 5.1, Example 5.2, Example 5.3 and Example 5.4 are $\mathbb{Z}_2$-taut. The pseudo-vertical surfaces presented in Example 5.5 are also $\mathbb{Z}_2$-taut. For the $\mathbb{Z}_2$-homology class represented by these surfaces, using the algorithm in \cite{NW,WY} and an online program given by \cite{Ka}, when the parameters $m,n,n_2,n_3$ are not too large, we find the lower bounds of the $\mathbb{Z}_2$-Thurston norm meet the exact values. For the example in \cite[Proposition 7.4]{LRS} ($M=S^2((3,-1),(4,1),(6,1))$ and $H_2(M;\mathbb{Z}_2)=\mathbb{Z}_2$), the lower bound is 1, the exact value is $N(4,1)+N(6,1)-2=3$.
\end{Remark}

\end{document}